\documentclass[10pt]{amsart}
\usepackage{graphicx}
\usepackage[centertags]{amsmath}
\usepackage{amsfonts}
\usepackage{amssymb}
\usepackage{amsthm}
\usepackage{newlfont}

\newtheorem{thm}{Theorem}[section]
\newtheorem{cor}[thm]{Corollary}
\newtheorem{lemma}[thm]{Lemma}
\newtheorem{prop}[thm]{Proposition}
\theoremstyle{definition}

\theoremstyle{remark}
\newtheorem{rem}[thm]{Remark}
\newcommand{\id}{\textrm{id}}

\begin{document}

\title{A Riemannian manifold with skew-circulant structures and an associated locally conformal K\"{a}hler manifold}
\author{Dimitar Razpopov}
\address{Dimitar Razpopov\\Department of Mathematics and Physics\\ Agricultural University of Plovdiv\\ 12 Mendeleev Blvd.\\ 4000 Plovdiv, Bulgaria} \email{razpopov@au-plovdiv.bg}
\author{Iva Dokuzova}
\address{Iva Dokuzova\\Department of Algebra and Geometry\\University of Plovdiv Paisii Hilendarski\\ 24 Tzar Asen, 4000 Plovdiv, Bulgaria} \email{dokuzova@uni-plovdiv.bg}

\begin{abstract}
A 4-dimensional Riemannian manifold $M$, equipped with an additional tensor structure $S$, whose fourth power is minus identity, is considered.
The structure $S$ has a skew-circulant matrix with respect to some basis of the tangent space at a point on $M$. Moreover, $S$ acts as an isometry with respect to the metric $g$. A fundamental tensor is defined on such a manifold $(M,g,S)$ by $g$ and by the covariant derivative of $S$. This tensor satisfies a characteristic identity which is invariant to the usual conformal transformation. Some curvature properties of $(M,g,S)$ are obtained.
A Lie group as a manifold of the considered type is constructed.
A Hermitian manifold associated with $(M,g,S)$ is also considered. It turns out that it is a locally conformal K\"{a}hler manifold.
\end{abstract}


\maketitle

\section{Introduction}\label{1}

 The classification of almost Hermitian manifolds with respect to the covariant derivative of the almost complex structure $J$ is made by Gray and Hervella in \cite{GrHer}.
 The Hermitian manifolds form a class of manifolds with an integrable almost complex structure $J$. Their subclass consists of the so-called locally conformal K\"{a}hler manifolds, determined  by a special property of the covariant derivative of $J$.  The class of K\"{a}hler manifolds is common to all classes in this classification and its manifolds have the richest geometry. The K\"{a}hler manifolds are extensively studied by many geometers. There is also a great interest in the study of locally conformal K\"{a}hler manifolds, since their structure $J$ satisfies similar but weaker conditions than those of K\"{a}hler manifolds. Some of the recent investigations of locally conformal K\"{a}hler manifolds are made in \cite{angella, cherevko, huang, moroianu-ornea, prvanovic, vilcu}.

Problems in differential geometry of a $4$-dimensional Riemannian manifold $M$ with a tensor structure $S$ of type $(1,1)$, which satisfies $S^{4}=-\id$, are considered in \cite{dok-raz}. The matrix of $S$ in some basis of the tangent space $T_{p}M$, at an arbitrary point $p$ on $M$, is skew-circulant. Moreover, $S$ is compatible with the metric $g$, so that an isometry is induced in $T_{p}M$. A Hermitian manifold $(M, g, J)$, where $J=S^{2}$, is associated with such a manifold $(M, g, S)$.

 In the present work we continue the study of $(M, g, S)$ and $(M, g, J)$.
  In Section~\ref{2}, we recall some necessary facts about these manifolds. In Section~\ref{4}, we compute the components of the fundamental tensor $F$ on $(M, g, S)$ determined by the metric $g$ and by the covariant derivative of $S$. We obtain an important characteristic identity for $F$. We establish that the image of the fundamental tensor with respect to the usual conformal transformation satisfies the same identity. In Section~\ref{5}, we find some curvature properties of $(M, g, S)$.  In Section~\ref{3}, we establish that the associated manifold $(M, g, J)$ belongs to the class of locally conformal K\"{a}hler manifolds. In Section~\ref{6}, we construct a Lie group with a Lie algebra of a special class as a manifold with the structure $(g, S)$.

 \section{Preliminaries}\label{2}
   We consider a $4$-dimensional Riemannian manifold $M$ equipped with a tensor $S$ of type $(1,1)$. The structure $S$ has a skew-circulant matrix, with respect to some basis $\{e_{i}\}$ of $T_{p}M$, given by
\begin{equation}\label{Q}
    (S_{j}^{k})=\begin{pmatrix}
      0 & 1 & 0 & 0\\
      0 & 0 & 1 & 0\\
      0 & 0 & 0 & 1\\
      -1 & 0 & 0 & 0\\
    \end{pmatrix}.
\end{equation}
Then $S$ has the property
\begin{equation*}
    S^{4}=-\id.
\end{equation*}
The metric $g$ and the structure $S$ satisfy
\begin{equation}\label{izometry}
 g(Sx, Sy)=g(x,y),\quad x, y\in \mathfrak{X}(M).
\end{equation}
The above condition and \eqref{Q} imply that the matrix of $g$ has the form:
\begin{equation}\label{metricg}
    (g_{ij})=\begin{pmatrix}
      A & B & 0 & -B \\
      B & A & B & 0\\
      0 & B & A & B\\
      -B & 0 & B & A\\
    \end{pmatrix}.
\end{equation}
Here $A$ and $B$ are smooth functions of an arbitrary point $p(x^{1}, x^{2},
x^{3}, x^{4})$ on $M$.
It is supposed that $A >\sqrt{2} B >0$ in order $g$ to be positive definite. The manifold $(M, g, S)$ is introduced in \cite{dok-raz}.

Anywhere in this work, $x, y, z, u$ will stand for arbitrary elements of the algebra of the smooth vector fields $\mathfrak{X}(M)$ or vectors in the tangent space $T_{p}M$. The Einstein summation convention is used, the range of the summation indices being always $\{1, 2, 3, 4\}$.

In \cite{dok-raz}, it is noted that the manifold $(M, g, J)$, where $J = S^{2}$, is a Hermitian manifold with an almost complex structure $J$. For such manifolds Grey-Hervella's classification is valid (\cite{GrHer}). This classification is made with respect to the covariant derivative of the K\"{a}hler form $J(x,y)=g(x,Jy).$ The almost Hermitian manifolds with an integrable structure $J$ are called Hermitian manifolds.
Their subclass of locally conformal K\"{a}hler manifolds is determined by the property
\begin{equation}\label{J}
\begin{split}
 g\big((\nabla_{x}J)y,z\big)=&\frac{1}{2}\big\{(g(x,y)\omega(z)-g(x,z)\omega(y)+g(x,Jy)\omega(Jz)\\&-g(x, Jz)\omega(Jy)\big\},\quad \omega(x)= g^{ij}g\big((\nabla_{e_{i}}J)e_{j},x\big).
\end{split}
\end{equation}
Here $\nabla$ is the Levi-Civita connection of $g$, and $g^{ij}$ are the components of the inverse matrix of $(g_{ij})$ with respect to the basis $\{e_{i}\}$ of $T_{p}M$.

It is known that the class of 4-dimensional locally conformal K\"{a}hler manifolds is non-trivial. Every K\"{a}hler manifold belongs to the class of locally conformal K\"{a}hler manifolds (\cite{GrHer}).

Now we consider an associated metric $\tilde{g}$ with $g$ on $(M, g, S)$, determined by
\begin{equation}\label{metricf}
 \tilde{g}(x, y)=g(x, Sy)+g(Sx, y).\end{equation}
The fundamental tensor $F$ of type $(0,3)$ and the 1-form $\theta$ are defined by
\begin{equation}\label{F}
  F(x,y,z)=(\nabla_{x}\tilde{g})(y,z),\quad \theta(x)=g^{ij}F(e_{i},e_{j},x),
\end{equation}
and $F$ has the property
\begin{equation}\label{prop-F}
  F(x,z,y)=F(x,y,z).
\end{equation}

If $F=0$, then the structure $S$ is parallel with respect to $\nabla$. The following necessary and sufficient conditions for $S$ and also for $J=S^{2}$ to be parallel structures with respect to $\nabla$ are established in \cite{dok-raz}.
\begin{thm}\label{alm-keler}
The manifold $(M, g, S)$ satisfies $\nabla S=0$ if and only if
\begin{align*}
   A_{1}=B_{2}-B_{4},\quad A_{2}=B_{1}+B_{3},\quad A_{3}=B_{2}+B_{4},\quad A_{4}=B_{3}-B_{1},
\end{align*}
where $A_{i}=\dfrac{\partial A}{\partial x^{i}}$, $B_{i}=\dfrac{\partial B}{\partial x^{i}}$.
\end{thm}
\begin{thm} The structure $S$ on $(M, g, S)$ satisfies $\nabla S=0$ if and only if $\nabla J=0$, i.e. $(M, g, J)$ is a K\"{a}hler manifold.
\end{thm}

\section{The fundamental tensor $F$ on $(M, g, S)$}\label{4}
In this section we obtain a characteristic property of the tensor $F$ on $(M, g, S)$, which is an analogue of the property \eqref{J} of $\nabla J$ on $(M, g, J)$.
For this purpose we calculate the components of $F$.
\begin{lemma}\label{tF}
 The components $F_{ijk}=F(e_{i}, e_{j}, e_{k})$ of the fundamental tensor $F$ on the manifold $(M, g, S)$ are given by
\begin{equation}\label{nablaf}
\begin{array}{ll}
 F_{111}=A_{2}-A_{4}-2B_{1}, & F_{112}=F_{114}=\frac{1}{2}(A_{3}-B_{2}-B_{4}),\\
 F_{122}=F_{322}=B_{1}+B_{3}-A_{2}, & F_{113}=-F_{324}=\frac{1}{2}(A_{2}+A_{4}-2B_{3}),\\
 F_{133}=F_{311}=0, & F_{134}= -F_{123}=\frac{1}{2}(A_{3}-B_{2}-B_{4}),\\
 F_{211}=-F_{411}=B_{2}-B_{4}-A_{1}, & F_{124}=F_{313}=\frac{1}{2}(A_{2}-2B_{1}-A_{4}),\\
 F_{222}=A_{1}-2B_{2}+ A_{3}, & F_{213}=-F_{424}=\frac{1}{2}(2B_{2}-A_{1}-A_{3}),\\
 F_{244}=F_{422}=0,& F_{212}=F_{214}=\frac{1}{2}(B_{3}-B_{1}-A_{4}),\\
F_{233}=F_{433}=B_{2}+B_{4}- A_{3}, & F_{234}=-F_{223}=\frac{1}{2}(B_{3}-B_{1}-A_{4}),\\
  F_{333}=-2B_{3}+A_{2}+A_{4}, & F_{224}=F_{413}=\frac{1}{2}(A_{3}-2B_{4}-A_{1}),\\
   F_{312}=F_{314}=\frac{1}{2}(B_{2}-B_{4}-A_{1}), & F_{334}=-F_{323}=\frac{1}{2}(B_{2}-B_{4}-A_{1}),\\
 F_{344}=-F_{144}=B_{3}-B_{1}- A_{4},  &  F_{412}=F_{414}=\frac{1}{2}(A_{2}-B_{1}-B_{3}),\\
F_{444}=-2B_{4}-A_{1}+ A_{3}, & F_{434}=-F_{423}=\frac{1}{2}(A_{2}-B_{1}-B_{3}).\\
\end{array}
\end{equation}
\end{lemma}
\begin{proof}
The inverse matrix of $(g_{ij})$ has the form:
\begin{equation}\label{g-obr}
    (g^{ik})=\frac{1}{D}\begin{pmatrix}
            A& -B & 0 & B\\
            -B& A& -B & 0\\
            0 &-B & A & -B\\
            B & 0 & -B & A\\
            \end{pmatrix},
  \end{equation}
where $D=A^{2}-2B^{2}$.

Using \eqref{Q} and \eqref{metricg} we get that the matrix $\tilde{g}$, determined by \eqref{metricf}, is
of the type:
\begin{equation}\label{tilde-g}
(\tilde{g}_{ij})=\begin{pmatrix}
      2B & A & 0 & -A\\
      A & 2B & A & 0\\
      0 & A & 2B & A\\
      -A & 0 & A & 2B\\
    \end{pmatrix}.
\end{equation}
Due to \eqref{F} the components of $F$ are $F_{ijk}=\nabla_{i}\tilde{g}_{jk}$.
We apply to $\tilde{g}$ the following well-known formula for the covariant derivative of tensors:
\begin{equation}\label{defF}
\nabla_{i}\tilde{g}_{jk}=\partial_{i}\tilde{g}_{jk}-\Gamma_{ij}^{a}\tilde{g}_{ak}-\Gamma_{ik}^{a}\tilde{g}_{aj}.
\end{equation}
Here $\Gamma_{ij}^{s}$ are the Christoffel symbols of $\nabla$. They are determined by
\begin{equation}\label{2.3}
2\Gamma_{ij}^{k}=g^{ak}(\partial_{i}g_{aj}+\partial_{j}g_{ai}-\partial_{a}g_{ij}).
\end{equation}
Then, with the help of \eqref{metricg}, \eqref{F}, \eqref{prop-F}, \eqref{g-obr}, \eqref{tilde-g} and \eqref{defF} we calculate the components of $F$, given in \eqref{nablaf}.
\end{proof}

Immediately, we have the following
\begin{cor}\label{t-theta}
 The components $\theta_{k}=g^{ij}F(e_{i}, e_{j}, e_{k})$ of the 1-form $\theta$ on the manifold $(M, g, S)$ are expressed by the equalities
 \begin{equation}\label{theta}
\begin{split}
\theta_{1}&=\frac{2}{D}\big(A(A_{2}-A_{4}-2B_{1})-2B(B_{2}-B_{4}-A_{1})\big),\\
\theta_{2}&=\frac{2}{D}\big(A(A_{1}+A_{3}-2B_{2})-2B(B_{1}+B_{3}-A_{2})\big),\\
\theta_{3}&=\frac{2}{D}\big(A(A_{2}+A_{4}-2B_{3})+2B(B_{2}+B_{4}-A_{3})\big),\\
\theta_{4}&=\frac{2}{D}\big(A(A_{3}-A_{1}-2B_{4})+2B(B_{3}-B_{1}-A_{4})\big).
\end{split}
\end{equation}
\end{cor}
\begin{proof}
  The proof follows from \eqref{nablaf} and \eqref{g-obr} by direct computations.
 \end{proof}
\begin{cor}\label{t-theta*}
 The components $\theta^{*}_{k}=g^{ij}F(e_{i}, Se_{j}, e_{k})$ of the 1-form $\theta^{*}$ on the manifold $(M, g, S)$ are expressed by the equalities
\begin{equation}\label{theta*}
\begin{split}
\theta^{*}_{1}&=\frac{2}{D}\big(A(B_{2}-B_{4}-A_{1})-B(A_{2}-A_{4}-2B_{1})\big),\\
\theta^{*}_{2}&=\frac{2}{D}\big(A(B_{1}+B_{3}-A_{2})-B(A_{1}+A_{3}-2B_{2})\big),\\
\theta^{*}_{3}&=\frac{2}{D}\big(A(B_{2}+B_{4}-A_{3})-B(A_{2}+A_{4}-2B_{3})\big),\\
\theta^{*}_{4}&=\frac{2}{D}\big(A(B_{3}-B_{1}-A_{4})-B(A_{3}-A_{1}-2B_{4})\big).
\end{split}
\end{equation}
 \end{cor}
 \begin{proof}
 Using \eqref{Q}, \eqref{nablaf} and \eqref{g-obr}, we find \eqref{theta*}.
 \end{proof}

Having in mind Lemma~\ref{tF}, Corollary~\ref{t-theta} and Corollary~\ref{t-theta*} we get the next statements.

\begin{thm}\label{tw1}
The fundamental tensor $F$ on the manifold $(M, g, S)$ satisfies the identity
\begin{equation}\label{c1}
\begin{split}
 F(x,y,z)=&\frac{1}{4}\big\{g(x,y)\theta(z)+g(x,z)\theta(y)+\big(g(Sx,y)+g(x, Sy)\big)\theta^{*}(z)\\&+\big(g(Sx,z)+g(x, Sz)\big)\theta^{*}(y)\big\}.
\end{split}
\end{equation}
\end{thm}
\begin{proof}
Using  \eqref{Q}, \eqref{metricg}, \eqref{nablaf},  \eqref{tilde-g}, \eqref{theta} and \eqref{theta*} we obtain
 \begin{equation}\label{usl-w1}
 \begin{split}
  F_{kij}=\frac{1}{4}\big(g_{kj}\theta_{i}+g_{ki}\theta_{j}+\tilde{g}_{kj}\theta^{*}_{i}+\tilde{g}_{ki}\theta^{^{*}}_{j}\big),\\
  \end{split}
\end{equation}
which is equivalent to \eqref{c1}.
\end{proof}
\begin{thm}\label{tw0}
The fundamental tensor $F$ on the manifold $(M, g, S)$ has the property
\begin{equation}\label{c3}
\begin{split}
 F(x,Jy,Jz)+F(y,Jz,Jx)+F(z,Jx,Jy)=0,
\end{split}
\end{equation}
where $J=S^{2}$.
\end{thm}
\begin{proof}
Due to \eqref{Q}, \eqref{F}, \eqref{prop-F} and \eqref{nablaf} we get that \eqref{c3} holds true.
\end{proof}
\begin{thm}\label{connF-lineF}
Under the conformal transformation
 \begin{equation}\label{conf}\overline{g}(x, y) = \alpha g(x, y),\end{equation}
where $\alpha$ is a smooth positive function, the tensor $F$ is transformed into the tensor
\begin{equation}\label{overlineF}
\begin{split}
 \overline{F}(x,y,z)=&\frac{1}{4}\big\{\overline{g}(x,y)\overline{\theta}(z)+\overline{g}(x,z)\overline{\theta}(y)+\big(\overline{g}(Sx,y)+\overline{g}(x, Sy)\big)\overline{\theta}^{*}(z)\\&+\big(\overline{g}(Sx,z)+\overline{g}(x, Sz)\big)\overline{\theta}^{*}(y)\big\}
 \end{split}
\end{equation}
with  $\overline{\theta}=\theta+\dfrac{2}{\alpha}\mathrm{d}\alpha\circ (S- S^{3})$ and $\overline{\theta}^{*}=\theta^{*}-\dfrac{2}{\alpha}\mathrm{d}\alpha.$
\end{thm}
\begin{proof}
 The inverse matrix of $(\tilde{g}_{ij})$ has the form
\begin{equation}\label{g2-obr}
    (\tilde{g}^{ik})=\frac{1}{2D}\begin{pmatrix}
            -2B& A & 0 & -A\\
            A& -2B& A & 0\\
            0 &A & -2B & A\\
            -A & 0 & A & -2B\\
            \end{pmatrix}.
  \end{equation}
Bearing in mind \eqref{metricg}, \eqref{g-obr},  \eqref{tilde-g} and \eqref{g2-obr}, we get
 \begin{equation}\label{fi-and-s}
 \tilde{g}_{ij}g^{is}=\Phi_{j}^{s}\ ,\quad g_{ij}\tilde{g}^{is}=\frac{1}{2}\Phi_{j}^{s},
 \end{equation}
 where
\begin{equation}\label{S-Phi}
(\Phi_{j}^{s})=\begin{pmatrix}
      0 & 1 & 0 & -1\\
      1 & 0 & 1 & 0 \\
      0 & 1 & 0 & 1\\
      -1 & 0 & 1 & 0 \\
    \end{pmatrix}.
    \end{equation}
    Because of  \eqref{Q} and \eqref{S-Phi} we have $\Phi=S-S^{3}$.

Now, from \eqref{theta}, \eqref{theta*} and \eqref{S-Phi}, we find
 \begin{equation}\label{theta-s*}
 \begin{split}
  \theta^{*}_{i}=-\frac{1}{2}\Phi_{i}^{s}\theta_{s}.
  \end{split}
\end{equation}

According to the transformation \eqref{conf}, the components of the tensor $\overline{F}$ are $\overline{F}_{ijk}=\overline{\nabla}_{i}\overline{\tilde{g}}_{jk}$, where $\overline{\tilde{g}}=\alpha \tilde{g}$ and $\overline{\nabla}$ is the Levi-Civita connection of $\overline{g}$. Therefore, it follows \begin{equation}\label{barF}\overline{\nabla}\ \overline{\tilde{g}}=\alpha\overline{\nabla}\tilde{g}+\tilde{g}\overline{\nabla}\alpha.\end{equation}

From the Christoffel formulas \eqref{2.3} and
\begin{equation*}
 2\overline{\Gamma}^{k}_{ij}=\overline{g}^{ks}(\partial_{i}\overline{g}_{sj}+\partial_{j}\overline{g}_{si}-\partial_{s}\overline{g}_{ij}),
 \end{equation*}
 and due to \eqref{conf} we get
 \begin{equation*}
 \overline{\Gamma}^{k}_{ij}= \Gamma^{k}_{ij}+\frac{1}{2\alpha}(\delta_{j}^{k}\alpha_{i}+\delta_{i}^{k}\alpha_{j}-g_{ij}g^{ks}\alpha_{s}),\quad
 \alpha_{s}=\frac{\partial\alpha}{\partial x^{s}}.
 \end{equation*}
 Then, applying \eqref{defF} to $\overline{\nabla}\tilde{g}$ and using \eqref{fi-and-s}, we obtain
 \begin{equation}\label{F-barF}
 \begin{split}
  \overline{\nabla}_{k}\tilde{g}_{ji}=\nabla_{k}\tilde{g}_{ji} &-\frac{1}{2\alpha}(\tilde{g}_{ji}\alpha_{k}+\tilde{g}_{ik}\alpha_{j}-g_{kj}\Phi_{i}^{s}\alpha_{s})\\&-\frac{1}{2\alpha}(\tilde{g}_{kj}\alpha_{i}+\tilde{g}_{ij}\alpha_{k}-g_{ik}\Phi_{j}^{s}\alpha_{s}).
\end{split}
\end{equation}
Substituting \eqref{F-barF} into \eqref{barF}, and taking into account \eqref{usl-w1}, \eqref{fi-and-s},  \eqref{S-Phi} and \eqref{theta-s*}, we get
\begin{equation*}
 \begin{split}
  \overline{\nabla}_{k}\overline{\tilde{g}}_{ji}=&\frac{1}{4}\big\{\alpha g_{kj}\big(\theta_{i}+\frac{2\alpha_{s}}{\alpha}\Phi_{i}^{s}\big)+\alpha g_{ki}\big(\theta_{j}+\frac{2\alpha_{s}}{\alpha}\Phi_{j}^{s}\big)\\&+\alpha \tilde{g}_{kj}\big(\theta^{*}_{i}-\frac{2\alpha_{i}}{\alpha}\big)+\alpha \tilde{g}_{ki}\big(\theta^{*}_{j}-\frac{2\alpha_{j}}{\alpha}\big)\big\},
\end{split}
\end{equation*}
which implies
 \begin{equation*}
 \begin{split}
  \overline{F}_{kji}=&\frac{1}{4}\big(\overline{g}_{kj}\overline{\theta}_{i}+\overline{g}_{ki}\overline{\theta}_{j}+\overline{\tilde{g}}_{kj}\overline{\theta}^{*}_{i}+\overline{\tilde{g}}_{ki}\overline{\theta}^{*}_{j}\big),\\&
   \overline{\theta}_{i}=\theta_{i}+\frac{2}{\alpha}\Phi_{i}^{s}\alpha_{s},\quad \overline{\theta}^{*}_{i}=-\frac{1}{2}\Phi_{i}^{s}\overline{\theta}_{s}.
\end{split}
\end{equation*}
Thus, for $(M, \overline{g}, S)$, the identity \eqref{overlineF} is valid .
\end{proof}

 \begin{rem} According to Theorem~\ref{connF-lineF}, we can say that $(M, g, S)$ and $(M, \overline{g}, S)$ belong to classes of the same type, defined by the equality \eqref{c1} for the corresponding metric.
\end{rem}
Immediately, from \eqref{metricf}, \eqref{F}, \eqref{overlineF} and \eqref{fi-and-s} it follows
\begin{cor}
 If $F=0$ holds, then it is valid
 \begin{equation}\label{W0-glob}
 \begin{split}
  \overline{F}(x, y,z)=&\frac{1}{2}\big\{g(x, y)\alpha(\Phi z)+g(x, z)\alpha(\Phi y)-\tilde{g}(x, y)\alpha(z)\\&-\tilde{g}(x, z)\alpha(y)\big\}.
  \end{split}
\end{equation}
\end{cor}
Next, we obtain
\begin{cor}
 If $F=0$ holds, then $\overline{F}$ vanishes if and only if $\alpha$ is a constant.
\end{cor}
\begin{proof} The local form of \eqref{W0-glob} is
\begin{equation}\label{W0}
  \overline{F}_{kij}=\frac{1}{2}\big(g_{kj}\Phi_{i}^{s}\alpha_{s}+g_{ki}\Phi_{j}^{s}\alpha_{s}-\tilde{g}_{kj}\alpha_{i}-\tilde{g}_{ki}\alpha_{j}\big).
\end{equation}
Let the tensor $\overline{F}$ vanish. Hence equality \eqref{W0} yields $$g_{kj}\Phi_{i}^{s}\alpha_{s}+g_{ki}\Phi_{j}^{s}\alpha_{s}-\tilde{g}_{kj}\alpha_{i}-\tilde{g}_{ki}\alpha_{j}=0.$$
 Contracting by $g^{kj}$ in the latter equality, and using \eqref{fi-and-s} and \eqref{S-Phi}, we find
    $\Phi_{i}^{s}\alpha_{s}=0$,
which implies $\alpha_{1}=\alpha_{2}=\alpha_{3}=\alpha_{4}=0$, i.e. $\alpha$ is a constant.

Vice versa. If $\alpha$ is a constant, then \eqref{W0} implies $\overline{F}=0$.
\end{proof}

\section{Some curvature properties of $(M, g, S)$}\label{5}
It is well-known, that the curvature tensor $R$ of $\nabla$ is defined by
\begin{equation*}
R(x, y)z=\nabla_{x}\nabla_{y}z-\nabla_{y}\nabla_{x}z-\nabla_{[x,y]}z.
\end{equation*}
The tensor of type $(0, 4)$ associated with $R$ is defined as follows:
\begin{equation*}
    R(x, y, z, u)=g(R(x, y)z,u).
\end{equation*}

The Ricci tensor $\rho$ and the scalar curvature $\tau$ with respect to $g$ are as usually:
\begin{equation}\label{def-rho}
    \rho(y,z)=g^{ij}R(e_{i}, y, z, e_{j}),\quad
    \tau=g^{ij}\rho(e_{i}, e_{j}).
\end{equation}

In this section we investigate some curvature properties of $(M, g, S)$, corresponding to the metric $g$ and to the associated metric $\tilde{g}$.

Let $\tilde{\Gamma}$ be the Christoffel symbols of $\tilde{g}$ and $\tilde{\nabla}$ the Levi-Civita connection of $\tilde{g}$. Let $\tilde{R}$
be the curvature tensor of $\tilde{\nabla}$. The Ricci tensor $\tilde{\rho}$ and the scalar curvature $\tilde{\tau}$ with respect to $\tilde{g}$ are given by
\begin{equation}\label{def-rho2}
    \tilde{\rho}(y,z)=\tilde{g}^{ij}\tilde{R}(e_{i}, y, z, e_{j}),\
    \tilde{\tau}=\tilde{g}^{ij}\tilde{\rho}(e_{i}, e_{j}).
\end{equation}

Let us denote
\begin{equation}\label{def-rho*}
    \tau^{*}=\tilde{g}^{ij}\rho(e_{i}, e_{j}),\quad \tilde{\tau}^{*}=g^{ij}\tilde{\rho}(e_{i}, e_{j}).
\end{equation}
Therefore we establish the following
\begin{thm}\label{connR-R}
 Let $\tilde{g}$ be the associated metric  with $g$ on $(M, g, S)$. For the Ricci tensors  $\rho$ and $\tilde{\rho}$ and for the scalar quantities $\tau$, $\tau^{*}$, $\tilde{\tau}$ and $\tilde{\tau}^{*}$ the following relation is valid:
 \begin{equation}\label{con-AE}
     \tilde{\rho}(x,y) = \rho(x,y)+\frac{1}{4}(\tilde{\tau}^{*}-\tau)g(x,y)+\frac{1}{4}(\tilde{\tau}-\tau^{*})\tilde{g}(x,y).
\end{equation}
\end{thm}
\begin{proof}
From \eqref{defF}, applying the Christoffel formulas \eqref{2.3} to $\Gamma$ and also to $\tilde{\Gamma}$, we obtain
\begin{equation*}
 \tilde{\Gamma}^{k}_{ij}=\Gamma^{k}_{ij}+\frac{1}{2}\tilde{g}^{ks}(\nabla_{i}\tilde{g}_{js}+\nabla_{j}\tilde{g}_{is}-\nabla_{s}\tilde{g}_{ij}).
 \end{equation*}
 Substituting \eqref{usl-w1} into the above equality, we get
  \begin{equation}\label{gamma4}
 \tilde{\Gamma}^{k}_{ij}=\Gamma^{k}_{ij}+\frac{1}{4}\tilde{g}^{ks}(g_{ij}\theta_{s}+\tilde{g}_{ij}\theta^{*}_{s}).
 \end{equation}

Using \eqref{fi-and-s}, \eqref{S-Phi} and \eqref{theta-s*} we find
  \begin{equation}\label{tild-g}
    \tilde{g}^{sk}\theta_{s}=-\theta^{*k},\quad \tilde{g}^{sk}\theta^{*}_{s}=-\frac{1}{2}\theta ^{k},\quad \tilde{g}_{sk}\theta^{s}=-2\theta^{*}_{k},\quad \tilde{g}_{sk}\theta^{*s}=-\theta_{k}.
\end{equation}
Bearing in mind \eqref{gamma4}, the first and the second equality of \eqref{tild-g}, we calculate the components of the tensor $\textrm{T}=\tilde{\Gamma}-\Gamma$ of the affine deformation. They are as follows:
 \begin{equation}\label{torsion}
 T^{k}_{ij}=-\frac{1}{4}\big(g_{ij}\theta^{*k}+\frac{1}{2}\tilde{g}_{ij}\theta^{k}\big).
 \end{equation}

For the components of the curvature tensors $\tilde{R}$ and $R$, it is well-known the relation
$$ \tilde{R}^{k}_{ijs} = R^{k}_{ijs} + \nabla_{j}T^{k}_{is}-\nabla_{s}T^{k}_{ij}
+T^{a}_{is}T^{k}_{aj}-T^{a}_{ij}T^{k}_{as}.$$
Then, taking into account \eqref{usl-w1}, \eqref{theta-s*}, \eqref{tild-g} and \eqref{torsion}, we
calculate
\begin{equation*}
\begin{split}
     \tilde{R}^{k}_{ijs} =& R^{k}_{ijs} - \frac{1}{4}g_{is}(\nabla_{j}\theta^{*k}-\frac{1}{4}\theta^{*}_{j}\theta^{*k})+\frac{1}{4}g_{ij}(\nabla_{s}\theta^{*k}-\frac{1}{4}\theta^{*}_{s}\theta^{*k})\\
 &-\frac{1}{8}\tilde{g}_{is}(\nabla_{j}\theta^{k}-\frac{1}{4}\theta_{j}\theta^{*k})+\frac{1}{8}\tilde{g}_{ij}(\nabla_{s}\theta^{k}-\frac{1}{4}\theta_{s}\theta^{*k}).
\end{split}
\end{equation*}
By contracting $k=s$ in the latter equality, and having in mind  \eqref{usl-w1}, \eqref{fi-and-s}, \eqref{theta-s*}, \eqref{def-rho}, \eqref{def-rho2} and \eqref{tild-g}, we get
\begin{equation}\label{tilde-S}
\begin{split}
     \tilde{\rho}_{ij} = \rho_{ij}+\frac{1}{4}g_{ij}\nabla_{s}\theta^{*s}+\frac{1}{4}\tilde{g}_{ij}\nabla_{s}\theta^{s}.
\end{split}
\end{equation}
Due to \eqref{fi-and-s}, \eqref{def-rho}, \eqref{def-rho2}, \eqref{def-rho*} and \eqref{tilde-S} we obtain
\begin{equation*}
     \tilde{\rho}_{ij} = \rho_{ij}+\frac{1}{4}(\tilde{\tau}^{*}-\tau)g_{ij}+\frac{1}{4}(\tilde{\tau}-\tau^{*})\tilde{g}_{ij},
\end{equation*}
which is a local form of \eqref{con-AE}.
\end{proof}

Further, we use the following statements for a special basis of $T_{p}M$ on $(M, g, S)$, which are established in \cite{dok-raz}.

(i) A basis of type $\{S^{3}x, S^{2}x, Sx, x\}$ of $T_{p}M$ exists and it is called an $S$-\textit{basis}. In this case we say that \textit{the vector $x$ induces an $S$-basis of} $T_{p}M$.

(ii) If a vector $x$ induces an $S$-basis and  $\varphi$ is the angle between $x$ and $Sx$, then \begin{equation}\label{g-cos}
\begin{split}
g(x, Sx)=g(x, x)\cos\varphi,\qquad \tilde{g}(x, x)=2g(x, x)\cos\varphi,
  \end{split}
 \end{equation}
and $\frac{\pi}{4}<\varphi<\frac{3\pi}{4}$.

(iii) An orthogonal $S$-\textit{basis}  of $T_{p}M$ exists.

Now, we recall that the Ricci curvature, with respect to $g$, in the direction of a nonzero vector $x$ is the value \begin{equation}\label{Ricicurv}
    r(x)=\frac{\rho(x,x)}{g(x,x)}.
\end{equation}
Due to Theorem~\ref{connR-R} we establish the following
\begin{cor}
 Let a vector $x$ induce an $S$-basis of $T_{p}M$ and let $\varphi$ be the angle between $x$ and $Sx$. If $r$ and $\tilde{r}$ are the Ricci curvatures in the direction of $x$ with respect to the metrics $g$ and $\tilde{g}$,  then
 \begin{equation}\label{con-r}
     \tilde{r}(x) = \frac{1}{2\cos\varphi}r(x)+\frac{1}{8\cos\varphi}(\tilde{\tau}^{*}-\tau)+\frac{1}{4}(\tilde{\tau}-\tau^{*}),\quad \varphi\neq\frac{\pi}{2}.
\end{equation}
\end{cor}
\begin{proof}
The proof follows directly from \eqref{con-AE}, \eqref{g-cos} and \eqref{Ricicurv}.
\end{proof}
In \cite{dok-raz}, a Riemannian manifold $(M, g, S)$ is called
almost Einstein if the metrics $g$ and $\tilde{g}$ satisfy
\begin{equation}\label{AE}\rho(x, y) = \beta g(x, y) + \gamma \tilde{g}(x, y),\end{equation} where $\beta$ and $\gamma$ are smooth functions on $M$.

It is known that a Riemannian manifold $(M, g)$ is called Einstein if the metric $g$ satisfies
\begin{equation}\label{E}\rho(x, y) = \beta g(x, y).\end{equation}

\begin{prop}\label{th2.4}
   Let the Levi-Civita connection $\tilde{\nabla}$ of $\tilde{g}$ be a locally flat connection on the manifold $(M, g, S)$. Then the following statements are valid.

   (i) $(M, g, S)$ is an almost Einstein manifold, and the Ricci tensor $\rho$ has the form
 \begin{equation}\label{rho-AE}
\rho(x, y) = \frac{\tau}{4}g(x, y)+\frac{\tau^{*}}{4}\tilde{g}(x,y).
\end{equation}

(ii) If a vector $x$ induces an $S$-basis, then the Ricci curvatures in the direction of the basis vectors are
\begin{equation}\label{ricc}
     r(x)=r(Sx)=r(S^{2}x)=r(S^{3}x)=\frac{\tau}{4}+\frac{\tau^{*}}{2}\cos\varphi ,
\end{equation}
where $\varphi=\angle(x, Sx)$.
\end{prop}
\begin{proof}
If $\tilde{\nabla}$ is a locally flat connection, then $\tilde{R}=0$. From \eqref{def-rho2} and \eqref{def-rho*} it follows $\tilde{\rho}=0$ and $\tilde{\tau}=\tilde{\tau}^{*}=0$. Hence \eqref{con-AE} implies \eqref{rho-AE}. Therefore, according to \eqref{AE}, we have that $(M, g, S)$ is an almost Einstein manifold.

Since $\rho$ is given by \eqref{rho-AE}, using \eqref{izometry} and \eqref{metricf}, we obtain
\begin{equation}\label{rhoL2}
\begin{split}
    \rho(x, x)=&\rho(Sx, Sx)=\rho(S^{2}x, S^{2}x)=\rho(S^{3}x, S^{3}x)\\&=\frac{\tau}{4}g(x,x)+\frac{\tau^{*}}{4}\tilde{g}(x,x).
    \end{split}
\end{equation}
Let a vector $x$ induce an $S$-basis. Hence equalities \eqref{g-cos}, \eqref{Ricicurv} and \eqref{rhoL2} imply \eqref{ricc}.
\end{proof}

\begin{cor}
Let $(M, g, S)$ be an Einstein manifold. If a vector $x$ induces an $S$-basis, then the Ricci curvatures in the direction of the basis vectors are
\begin{equation*}
     r(x)=r(Sx)=r(S^{2}x)=r(S^{3}x)=\frac{\tau}{4}.
\end{equation*}
\end{cor}
\begin{proof}
By comparing \eqref{rho-AE} and \eqref{E} we have that $\tau^{*}$ vanishes. Thus the above equalities follow directly by substituting $\tau^{*}=0$ into \eqref{ricc}.
\end{proof}

In a similar way to Proposition~\ref{th2.4} we prove the next
\begin{prop}
 Let the Levi-Civita connection $\nabla$ of $g$ be a locally flat connection on the manifold $(M, \tilde{g}, S)$. Then the following statements are valid.

   (i) $(M, \tilde{g}, S)$ is an almost Einstein manifold and the Ricci tensor $\tilde{\rho}$ has the form
$$\tilde{\rho}(x, y) = \frac{\tilde{\tau}}{4}\tilde{g}(x,y)+\frac{\tilde{\tau}^{*}}{4}g(x, y).$$

(ii) If a vector $x$ induces an $S$-basis, then the Ricci curvatures in the direction of the basis vectors are
$$\tilde{r}(x)=\tilde{r}(Sx)=\tilde{r}(S^{2}x)=\tilde{r}(S^{3}x)=\frac{\tilde{\tau}}{4}+\frac{\tilde{\tau}^{*}}{8\cos\varphi},\quad \varphi\neq\frac{\pi}{2}.$$
\end{prop}

\section{A locally conformal K\"{a}hler manifold $(M, g, J)$}\label{3}
The fundamental K\"{a}hler form of the structure $(g, J)$ on an almost complex manifold $(M, g, J)$ is determined by
\begin{equation}\label{tensorj}
J(x,y)=g(x,Jy)
\end{equation}
and it is skew-symmetric, i.e. $J(x,y)=-J(y,x)$ (\cite{GrHer}).

In this section we consider a Hermitian manifold $(M, g, J)$ with a complex structure $J=S^{2}$.
\begin{lemma}\label{tj}
 The nonzero components $\nabla_{i}J_{jk}=g((\nabla_{e_{i}}J)e_{j}, e_{k})$ of the fundamental tensor $\nabla J$ on the manifold $(M, g, J)$ are given by
\begin{equation}\label{nablaF}
\begin{array}{ll}
 \nabla_{3}J_{12}=-\nabla_{3}J_{34}=\nabla_{1}J_{23}=-\nabla_{1}J_{14}=\frac{1}{2}(B_{1}+B_{3}-A_{2}),\\
 \nabla_{1}J_{34}=-\nabla_{1}J_{12}=\nabla_{3}J_{23}=-\nabla_{3}J_{14}=\frac{1}{2}(A_{4}+B_{1}-B_{3}),\\
 \nabla_{2}J_{34}=-\nabla_{2}J_{12}=\nabla_{4}J_{23}=-\nabla_{4}J_{14}=\frac{1}{2}(B_{2}+B_{4}-A_{3}),\\
 \nabla_{4}J_{12}=-\nabla_{4}J_{34}=\nabla_{2}J_{23}=-\nabla_{2}J_{14}=\frac{1}{2}(A_{1}+B_{4}-B_{2}).\\
\end{array}
\end{equation}
\end{lemma}
\begin{proof}
Because of \eqref{Q}, \eqref{metricg} and \eqref{tensorj}, we get that the matrix of the fundamental K\"{a}hler form is
of the type:
\begin{equation}\label{j}
  (J_{ik})=\begin{pmatrix}
            0& B & A & B\\
            -B& 0& B & A\\
            -A &-B & 0 & B\\
            -B & -A & -B & 0\\
            \end{pmatrix}.
\end{equation}
Applying the Christoffel symbols $\Gamma$, obtained by \eqref{metricg}, \eqref{g-obr} and  \eqref{2.3}, and the components of the matrix \eqref{j} to equality
\begin{equation*}
\nabla_{i}J_{jk}=\partial_{i}J_{jk}-\Gamma_{ij}^{a}J_{ak}-\Gamma_{ik}^{a}J_{aj},
\end{equation*}
we calculate the nonzero components of $\nabla J$, given in \eqref{nablaF}.
\end{proof}

Immediately, we have the following
\begin{cor}\label{lema-omega}
 The components $\omega_{k}=g^{ij}F(e_{i}, e_{j}, e_{k})$ of the 1-form $\omega$ on the manifold $(M, g, J)$ are expressed by the equalities
 \begin{equation}\label{omega}
\begin{split}
\omega_{1}&=\frac{1}{D}\big(A(B_{4}+B_{2}-A_{3})+B(2B_{3}-A_{2}-A_{4})\big),\\
\omega_{2}&=\frac{1}{D}\big(A(B_{3}-B_{1}-A_{4})+B(2B_{4}+A_{1}-A_{3})\big),\\
\omega_{3}&=\frac{1}{D}\big(A(A_{1}+B_{4}-B_{2})+B(-2B_{1}+A_{2}-A_{4})\big),\\
\omega_{4}&=\frac{1}{D}\big(A(A_{2}-B_{1}-B_{3})+B(-2B_{2}+A_{1}+A_{3})\big),\\
\end{split}
\end{equation}
\end{cor}
\begin{proof}
  The equalities \eqref{omega} follow from \eqref{g-obr} and \eqref{nablaF} by direct computations.
 \end{proof}
Due to Lemma~\ref{tj} and Corollary~\ref{lema-omega} we establish the following
\begin{thm}
The manifold $(M, g, J)$ is a locally conformal K\"{a}hler manifold.
\end{thm}

\begin{proof}
Using  \eqref{Q}, \eqref{metricg},  \eqref{j}, \eqref{omega} and Lemma~\ref{tj} we obtain
 \begin{equation}\label{c2}
 \begin{split}
  \nabla_{k}J_{ij}=&\frac{1}{2}\big(g_{ki}\omega_{j}-g_{kj}\omega_{i}+J_{ki}\tilde{\omega}_{j}-J_{kj}\tilde{\omega}_{i}\big),\quad \tilde{\omega}_{i}=J_{i}^{a}\omega_{a}.
  \end{split}
\end{equation}
 The latter identity is the local form of \eqref{J}, which is
 a defining condition of a locally conformal K\"{a}hler manifold.
\end{proof}

\section{A Lie group with a structure $(g, S)$}\label{6}

Let $G$ be a $4$-dimensional real connected Lie group. Let $\mathfrak{g}$ be the corresponding Lie algebra with a basis $\{e_{1}, e_{2},e_{3},e_{4}\}$ of left invariant vector fields. We introduce a skew-circulant structure $S$ and a metric $g$ as follows:
\begin{eqnarray}
&Se_{1}=-e_{4},\quad Se_{2}=e_{1},\quad Se_{3}=e_{2},\quad Se_{4}=e_{3};\label{lie} \\
&g(e_{i}, e_{j})= \delta_{ij},\label{g}
\end{eqnarray}
where $\delta_{ij}$ is the Kronecker delta.

Consequently the used basis $\{e_i\}$ is an orthonormal $S$-basis. Obviously, \eqref{Q} and \eqref{izometry} are valid and $(g, S)$ is a structure of the considered type. We denote the corresponding manifold by $(G, g, S)$. Then the associated manifold is $(G, g, J)$, where $J$ satisfies
\begin{equation}\label{lie-j}
Je_{1}=-e_{3},\quad Je_{2}=-e_{4},\quad Je_{3}=e_{1},\quad Je_{4}=e_{2}.
\end{equation}

The real four-dimensional indecomposable Lie algebras are classified by Mubarak\-zyanov (\cite{mub}). This scheme seems to be the most popular (see \cite{Biggs} and the references therein).
We pay attention to the class $\{\mathfrak{g}_{4,5}\}$, which represents an indecomposable Lie algebra,
depending on two real parameters $a$ and $b$. Actually, it induces
a family of manifolds whose properties are functions of $a$ and $b$.

According to the definition of the class $\{\mathfrak{g}_{4,5}\}$, we have that the nonzero brackets are as follows (\cite{Biggs}):
\begin{equation}\label{skobki1}
  [e_{1}, e_{4}]=e_{1},\; [e_{2}, e_{4}]=ae_{2},\; [e_{3}, e_{4}]=be_{3},\quad
  -1\leq b \leq a \leq 1,\; ab\neq 0.
\end{equation}
The well-known Koszul formula implies
\begin{equation*}
    2g(\nabla_{e_{i}}e_{j}, e_{k})=g([e_{i}, e_{j}],e_{k})+g([e_{k}, e_{i}],e_{j})+g([e_{k}, e_{j}],e_{i})
\end{equation*}
 and, using \eqref{g} and \eqref{skobki1}, we obtain
\begin{equation}\label{nabla-e}
\begin{array}{lll}
    \nabla_{e_{1}}e_{1}=-e_{4},\quad
&    \nabla_{e_{1}}e_{4}=e_{1},\quad
&    \nabla_{e_{2}}e_{2}=-ae_{4},\\
   \nabla_{e_{2}}e_{4}=ae_{2},\quad
&    \nabla_{e_{3}}e_{4}=be_{3},\quad
&    \nabla_{e_{3}}e_{3}=-be_{4}.
\end{array}
\end{equation}
Furthermore, with the help of the above equalities we compute the components of the tensor $F$ on $(G, g, S)$ and the components of the tensor $\nabla J$ on $(G, g, J)$. We find conditions under which $F$ satisfies \eqref{c1} and $\nabla J$ satisfies \eqref{J}.
\begin{prop}\label{kt2}
If $\mathfrak{g}$ belongs to $\{\mathfrak{g}_{4, 5}\}$, then the fundamental tensor $F$ on $(G, g, S)$ satisfies the property \eqref{c1} if and only if the condition
$a=b=1$ holds.
\end{prop}
\begin{proof}
Bearing in mind \eqref{F}, \eqref{lie}, \eqref{g} and \eqref{nabla-e}
we get the
components $F_{ijk}$ of $F$, $\theta_i$ of $\theta$ and  $\theta^{*}_i$ of $\theta^{*}$ with respect to the basis $\{e_i\}$.
The nonzero of them are the following:
\begin{equation}\label{F45}
\begin{array}{ll}
  F_{124}=-F_{113}=-F_{134}= -1,& F_{111}=-F_{144}=-2,\\  F_{313}=F_{332}=F_{324}=-b,&
  F_{333} =-F_{344} =2b,\\ F_{212}=F_{214}=-F_{223}=F_{234}=-a.
\end{array}
\end{equation}
\begin{equation}\label{theta-lie}
\begin{array}{ll}
 \theta_{1}=-2-a-b,& \theta_{3}=2a+b+1,\\
 \theta^{*}_{2}=\frac{1}{2}(1-b), &\theta^{*}_{4}=-\frac{1}{2}(2a+3b+3).
 \end{array}
\end{equation}

By equalities \eqref{metricf}, \eqref{lie} and \eqref{g},
we find
\begin{equation}\label{tildeg}
\begin{array}{l}
\tilde{g}(e_{1}, e_{1})= \tilde{g}(e_{2}, e_{2})=\tilde{g}(e_{3}, e_{3})=\tilde{g}(e_{4}, e_{4})=0,\\ \tilde{g}(e_{1}, e_{3})=\tilde{g}(e_{2}, e_{4})= 0,\\
\tilde{g}(e_{1}, e_{2})= \tilde{g}(e_{2}, e_{3})=\tilde{g}(e_{3}, e_{4})=-\tilde{g}(e_{1}, e_{4})=1.
\end{array}
\end{equation}
Hence \eqref{lie}, \eqref{g}, \eqref{F45}, \eqref{theta-lie} and \eqref{tildeg} imply that the condition
\eqref{usl-w1} holds if and only if $a=b=1$.
\end{proof}

\begin{prop}\label{W1}
  If $\mathfrak{g}$ belongs to $\{\mathfrak{g}_{4,5}\}$, then
$(G, g, J)$ belongs to the class of locally conformal K\"{a}hler manifolds if and only if the condition $a=b=1$ holds.
\end{prop}
\begin{proof}
Bearing in mind \eqref{tensorj}, \eqref{g}, \eqref{lie-j} and \eqref{nabla-e}
we obtain the
components $\nabla_{i}J_{jk}$ of $\nabla J$
and $\omega_i$ of $\omega$ with respect to the basis $\{e_i\}$.
The nonzero of them are the following:
\begin{equation}\label{J45}
\begin{split}
  \nabla_{1}J_{12}=-\nabla_{1}J_{34}=1,\quad
  \nabla_{3}J_{32}=\nabla_{3}J_{14}=b, \quad  \omega_{2}=b+1.
\end{split}
\end{equation}
By equalities \eqref{tensorj}, \eqref{lie}, \eqref{g} and \eqref{J45}, we get that the condition
\eqref{c2} holds if and only if $a=b=1$.
\end{proof}
\begin{rem}
Obviously, if $\mathfrak{g}$ is in $\{\mathfrak{g}_{4,5}\}$, then $(G, g, J)$ is not a K\"{a}hler manifold.
\end{rem}

By virtue of Proposition~\ref{kt2} and Proposition~\ref{W1}, we immediately have the following
\begin{cor}
If $\mathfrak{g}$ belongs to $\{\mathfrak{g}_{4, 5}\}$, then the fundamental tensor $F$ on $(G, g, S)$  satisfies \eqref{c1} if and only if $(G, g, J)$ belongs to the class of locally conformal K\"{a}hler manifolds.
\end{cor}

Next we get
\begin{prop}\label{kt3}
 Let $(G, g, S)$ be a manifold with a Lie algebra $\mathfrak{g}$ from the class $\{\mathfrak{g}_{4, 5}\}$. If $a=b=1$ are valid, then $(G, g, S)$ is a non-flat Einstein manifold with a negative scalar curvature $\tau=-12$.
\end{prop}
\begin{proof}
We calculate the components $R_{ijks}$ of the curvature tensor $R$ with respect to $\{e_i\}$, having in mind  the symmetries of $R$ and the condition $a=b=1$. The nonzero of them are
\begin{align}\label{rlamda}
  R_{1212}=R_{1414}=R_{2323}=R_{3434}=R_{1313}=R_{2424}=1.
\end{align}
Using \eqref{def-rho} and \eqref{rlamda}, we compute the components of $\rho$ and the value of $\tau$. The nonzero of them are as follows:
\begin{equation*}
    \rho_{11}=\rho_{22}=\rho_{33}=\rho_{44}=-3,\qquad
    \tau=-12.
\end{equation*}
Then, from \eqref{g}, we get $\rho=\frac{\tau}{4}g$. Consequently, due to \eqref{E}, the manifold $(G, g, S)$ is Einstein.

\end{proof}
\begin{rem} Lie groups with a Lie algebra in $\{\mathfrak{g}_{4,5}\}$ are studied in \cite{dok-raz-man} as an example of 4-dimensional Riemannian manifolds with circulant structures.

 An example of a locally conformal K\"{a}hler manifold constructed on a Lie group with a Lie algebra in $\{\mathfrak{g}_{4, 5}\}$ is considered in \cite{HM}. The metric of the manifold is indefinite but the condition $a=b=1$ also exists.

The example of a 4-dimensional Riemannian manifold with skew-circulant structures, constructed in this section, has similar properties to the above examples.
\end{rem}



\begin{thebibliography}{99}
\bibitem{angella} D. Angella, M. Origlia, \emph{Locally conformally Kähler structures on
four-dimensional solvable Lie algebras}, Complex Manifolds {\bf 7} (2020),  1--35.

\bibitem{Biggs} R.~Biggs, C.~C.~Remsing, \emph{On the classification of real four-dimensional Lie groups}, J. Lie Theory {\bf 26} (2016), no. 4, 1001--1035.

\bibitem{cherevko} Y.~Cherevko, V.~Berezovski, I.~Hinterleitner, D.~Smetanova, \emph{Infinitesimal transformations of locally conformal K\"{a}hler manifolds}, Mathematics (2019) {\bf 7}(8), 658.  doi.org/10.3390/math7080658.

\bibitem{dok-raz}
I.~Dokuzova, D.~Razpopov, \emph{Four-dimensional almost Einstein manifolds with skew-circulant structures}, J. Geom. {\bf 111} (2020), no. 9, 18 pp.

\bibitem{dok-raz-man}
I.~Dokuzova, D.~Razpopov, M.~Manev, \emph{Two types of Lie groups as 4-dimensional Riemannian manifolds with circulant structure}, Math. Educ. Math. {\bf 47}, in: Proc. 47th Spring Conf. UBM, Borovets, (2018), 115--120.

\bibitem{GrHer}
A.~Gray, L.M.~Hervella, \emph{The sixteen classes of almost Hermitian manifolds and their linear invariants}, Ann. Mat. Pura Appl. (4), {\bf 123} (1980),  35--58.

\bibitem{huang}
T.~Huang,
\emph{A note on Euler number of locally conformally K\"{a}hler manifolds}, Math. Z. (2020), doi.org/10.1007/s00209-020-02491-y.

\bibitem{HM} H.~Manev, \emph{Almost hypercomplex manifolds with Hermitian-Norden metrics and 4-dimensional indecomposable real Lie algebras depending on two parameters}, C. R. Acad. Bulgare Sci. {\bf 73} (2020),  no. 5,  589-599.

\bibitem{moroianu-ornea} A.~Moroianu, S.~Moroianu, L.~Ornea, \emph{Locally conformally Kähler manifolds with holomorphic Lee field}, Diff. Geom. Appl., {\bf 60} (2018), 33-38.

\bibitem{mub}
G.~M.~Mubarakzyanov, \emph{On solvable Lie algebras}, Izv. Vys\v{s}. U\v{c}ebn. Zaved. Matematika (1963), no. 1 (32), 114--123. (in Russian)

\bibitem{prvanovic}
M.~Prvanovi\'{c}, \emph{Some properties of the locally conformal K\"{a}hler manifold},  Bull. Cl. Sci. Math. Nat. Sci. Math. {\bf 35} (2010), 9--23.

\bibitem{vilcu} G.~Vilcu, \emph{Ruled CR-submanifolds locally conformal K\"{a}hler manifolds}, J. Geom. Phys. {\bf 62} (2012), no. 6,  1366--1372.




\end{thebibliography}

\end{document}